\documentclass[10pt]{amsart}

\usepackage[OT1]{fontenc}
\usepackage{amsmath,mathrsfs, amsfonts, amssymb, enumerate, caption, subcaption}
\usepackage[numbers]{natbib}

\usepackage{nameref}
\usepackage{zref-xr,zref-user}
\zxrsetup{toltxlabel}

\usepackage{tikz}
\usetikzlibrary{arrows, calc}
\usetikzlibrary{patterns, shapes, decorations.markings}

\usepackage[colorlinks,citecolor=blue,urlcolor=blue]{hyperref}
\usepackage{pgfplots}
\pgfplotsset{/pgf/number format/use comma,compat=newest}
\usetikzlibrary{external}
\numberwithin{equation}{section}

\numberwithin{equation}{section}
\theoremstyle{definition}
\newtheorem{definition}[equation]{Definition}
\theoremstyle{definition}
\newtheorem{remark}[equation]{Remark}
\theoremstyle{definition}

\theoremstyle{definition}

\theoremstyle{lemma}

\theoremstyle{lemma}
\newtheorem{lemma}[equation]{Lemma}
\theoremstyle{theorem}
\newtheorem{theorem}[equation]{Theorem}
\theoremstyle{proposition}
\newtheorem{proposition}[equation]{Proposition}
\theoremstyle{corollary}
\newtheorem{corollary}[equation]{Corollary}

\theoremstyle{corollary}
\theoremstyle{definition}
\newtheorem{example}[equation]{Example}
\theoremstyle{example}

\theoremstyle{proposition}

\theoremstyle{definition}

\DeclareMathOperator*{\supp}{supp}

\newcommand{\one}{\mathbf{1}}

\newcommand{\R}{\mathbf{R}}
\newcommand{\PP}{\mathbf{P}}

\newcommand{\E}{\mathbf{E}}
\renewcommand{\P}{\mathbf{P}}

\newcommand{\N}{\mathbf{N}}

\renewcommand{\b}{b}

\title[Langevin dynamics with singular potentials]{Ergodicity and Lyapunov functions for Langevin dynamics with singular potentials}

\author[D.P. Herzog]{David P. Herzog}
\address{Department of Mathematics, Iowa State University, Ames, IA 50011}
\author[J. C. Mattingly]{Jonathan C. Mattingly}
\address{Department of Mathematics, Center for Theoretical
    and Mathematical Sciences,
Center for Nonlinear and Complex Systems, and Department of Statistical
    Sciences\\ Duke
    University, Durham, NC, 27708-0251}

\begin{document}
\maketitle

\begin{abstract}
We study Langevin dynamics of $N$ particles on $\R^d$ interacting through a singular repulsive potential, e.g.~the well-known Lennard-Jones type, and show that the system converges to the unique invariant Gibbs measure exponentially fast in a weighted total variation distance.  The proof of the main result relies on an explicit construction of a Lyapunov function.  In contrast to previous results for such systems~\cite{CG_10,Grothaus_Stilgenbauer_2015}, our result implies geometric convergence to equilibrium starting from an essentially optimal family of initial distributions.                    
\end{abstract}

\section{Introduction}

We are interested in second-order Langevin dynamics when the potential
is singular. Important examples of dynamics in this class include
classical models of molecules interacting through a Lennard-Jones
potential or massive particles interacting through a Coulomb
potential. Despite its apparent applicability to physical systems, the
mathematical theory underlying these dynamics, in particular how they
relax to equilibrium, is less understood.  The primary aim of this
paper is to prove that such systems converge exponentially fast to the
unique invariant Gibbs measure when started from a wide array of initial distributions.

Throughout, we consider $N\geq 1$ particles evolving on Euclidean space $\R^d$ according to the Langevin stochastic differential equation
\begin{equation}  \label{eqn:sdem}
\left\{
  \begin{aligned}
    dq(t)&= p(t) \, dt \\
  dp( t)&= [-\gamma p(t) - \nabla U(q(t))] \, dt + \sqrt{2\gamma T} \, dB(t) 
  \end{aligned}
\right.
\end{equation}
where $q(t)=(q_1(t), \ldots, q_N(t)) \in (\R^d)^N$ and $p(t)=(p_1(t),
\ldots, p_N(t))\in(\R^d)^N$ denote, respectively, the positions
$q_i\in\R^d$ and momenta $p_i\in\R^d$ of the particles. $B(t)=(B_1(t), \ldots,B_N(t))$ is a vector of mutually independent
$d$-dimensional Brownian Motions defined on a probability space $(\Omega, \mathcal{F}, \P)$.  The constant $\gamma>0$ governs
the intensity of the friction of the thermal medium while the constant
$T>0$ is the temperature.

The potential $U\colon(\R^d)^N\rightarrow [0, +\infty]$  defines the
interactions between the particles as well as any external
potential forces. For definiteness, the reader may find it useful to consider the
following singular example, which was one of our principle
motivations for this work:
\begin{align}
\label{eqn:LJP}
  U(q) = \sum_{i=1}^N U_{0}(q_i) + \sum_{i <j  } U_{I}(
  q_i - q_j).
\end{align}
In the expression above, the confining part $U_0$ is due to external forces
keeping the particles in the domain of definition, and is, for
example, often taken to be a polynomial growing at least as fast as a
quadratic such as $|q|^6$ or $(1-|q|^2)^2$.  The interaction part $U_I$ describes repulsive forces as particles approach one another at close range.  For example, one may consider the classical
\emph{Lennard-Jones} interaction of
\begin{align}
  U_I(q) = \frac{c_0}{|q|^{12}} -  \frac{c_1}{|q|^6}
\end{align}
for some positive constants $c_0$ and $c_1$.

So long as the potential $U$ is $C^2$ away from its singularities, one can use the Hamiltonian  
\begin{align*}
  H(p,q)=\frac{|p|^2}{2}  + U(q)
\end{align*}
to show that pathwise solutions of~\eqref{eqn:sdem} exist and are
unique for all finite times $t\geq 0$.  Indeed, pathwise solutions
exist and are unique up until the (random) time $\tau$ at which the
process exits the domain of definition.  The Hamiltonian then serves
as a basic type of Lyapunov function allowing one to conclude that
$\tau= \infty$ almost surely.  Extracting a convergence rate to
equilibrium, on the other hand, requires more structure and work, as
$H$ alone does not give the needed dissipative bound since the expected
rate of change of $H$ cannot be controlled by a function of $H$ itself.    

To get around this issue, if $U$ is absent of singularities, one can apply the perturbative `$pq$ trick', adding a term of the form $c\, p \cdot q$ for some $c>0$ to the Hamiltonian~\cite{MSH_02,Talay_2002}.  This results in a Lyapunov function $V$ of the form $V= H + c\, p \cdot q$ that ultimately implies geometric convergence to equilibrium in such cases.  Due to the presence of singularities in $U$, however, this simple trick does not work, and thus understanding how the dynamics~\eqref{eqn:sdem} converges to equilibrium is more involved.  Consequently, the system~\eqref{eqn:sdem} with a singular potential $U$ is not covered by most results on second-order Langevin dynamics. Notable exceptions are the
the papers~\cite{CG_10,Grothaus_Stilgenbauer_2015} where a mixture of martingale techniques and ideas from hypocoercivity are used to prove the existence of the process
as well as exponential mixing, in the sense of exponential
decoration, provided the system is started in the equilibrium measure $\mu$ given by 
\begin{align}\label{def:InvMeasure}
  \mu(dp \times dp) = \frac{1}{\mathcal{Z}} e^{-\frac1T H(p,q)} \, dq \times dp. 
\end{align}
In the above, $\mathcal{Z}>0$ is the normalization constant which makes $\mu$ a
probability measure.  

In contrast to these works, the authors in
\cite{cooke17:_geomet_ergod_two} establish the existence and exponential
convergence to $\mu$ starting from arbitrary
initial conditions in an appropriate weighted total variation distance
(see~\cite{HM_08, MTIII}). However, the setting of this paper was limited to a single particle in dimension $d=1$ interacting with the origin via a Lennard-Jones-like potential.  In this note, we (finally) succeed in producing the results we have
desired from the onset. In particular, we prove that for a wide class of potentials, which we call \emph{admissible}, the stochastic dynamics~\eqref{eqn:sdem} converges exponentially fast to $\mu$ in a weighted total
variation norm. This is done by constructing an explicit Lyapunov function.  The class of admissible potentials, defined below in
in Section~\ref{sec:main-results}, is comparable to those considered in
\cite{CG_10,Grothaus_Stilgenbauer_2015}, but the results are in line
with those proven in \cite{cooke17:_geomet_ergod_two} and hence stronger.
In particular, our convergence results allow for
the analysis of numerical methods used to simulate molecular dynamics
or sample from the density using Monte Carlo methods.

\section{Main Results}
\label{sec:main-results}

We begin by defining the class of \emph{admissible} potentials $U$.  Implicit in the definition of $\mu$ in \eqref{def:InvMeasure} is that $\mu$ is
normalizable which, in light of the definition of $H$, is equivalent to 
\begin{align}
  \label{eq:UIntegrable}
  \int_\mathcal{O} e^{-\frac1T U(q)} dq < \infty
\end{align}
where
\begin{align}
  \label{eq:Odef}
  \mathcal{O}= \{ q\in (\R^d)^N \, : \, U(q)<  \infty\}\,.
\end{align}

\begin{definition}\label{def:admissible}
We say that a function $U\colon(\R^d)^N\rightarrow [0, + \infty]$ is an
\emph{admissible potential}  if it  is satisfies the normalization 
condition \eqref{eq:UIntegrable} and the following regularity and
growth conditions:
\begin{itemize}
\item  $U\in C^\infty(\mathcal{O})$.  
\item  $\mathcal{O}$ is an open, path connected set.  Moreover for each $R>0$, the set 
\begin{align*}
\{ q\in (\R^d)^N \, : \,
U(q)< R\}
\end{align*}
is open and \emph{precompact}, i.e. has compact closure, in $(\R^d)^N$.
\item For any sequence $\{q_k\} \subseteq \mathcal{O}$ with
  $U(q_k)\rightarrow \infty$,  we have the following asymptotic properties: 
\begin{align*}
|\nabla U(q_k)| \rightarrow \infty\, \,\,\, \text{ and } \,\,\,\, \frac{|\nabla^2 U(q_k)|}{|\nabla U(q_k)|^2}\rightarrow 0 
\end{align*}
where $\nabla^2 $ denotes the Hessian operator. 
\end{itemize}
\end{definition}

For the remander of this note, we will assume that the potential $U$
is admissible and we will only consider dynamics on the state space $\mathcal{X}=
\mathcal{O} \times (\R^d)^N$.  We remark that if the set $\mathcal{O}$
is not pathwise connected then the dynamics will be restricted to each
connected component for all time and the theorems apply to each
connected component by setting $U=\infty$ on the other connected components.

We now state the following proposition giving global existence and
uniqueness of the process on $\mathcal{X}$.  Its proof is a simple
consequence of the proof of the main result (Theorem~\ref{thm:main})
below, though it could also be proven by considering only $H$ rather
than the Lyapunov function we will eventually construct.  
\begin{proposition}
\label{prop:exist} 
For every initial condition $x(0)=(q(0), p(0)) \in \mathcal{X}$, equation 
~\eqref{eqn:sdem} has a unique pathwise solution $x(t)=(q(t), p(t))$. Furthermore,  the solution 
remains in  $\mathcal{X}$ for all 
finite times $t\geq 0$ almost surely. 
\end{proposition}
With Proposition~\ref{prop:exist} guaranteeing global
existence in time, we can define the Markov operator $P_t$ associated to
\eqref{eqn:sdem} by
\begin{align*}
  (P_t\phi)(q,p) = \E_{(q,p)} \phi(q(t),p(t))
\end{align*}
for all $(q,p) \in \mathcal{X}$, $t \geq 0$ and $\phi \colon
\mathcal{X} \rightarrow \R$ bounded, measurable. This action on test
functions induces the standard dual action, denoted $\nu P_t$, on probability measures
$\nu$ on Borel subsets $A$ of $\mathcal{X}$ through
\begin{align*}
  (\nu P_t)(A)  = \int_\mathcal{X} (P_t \one_A)(q,p) \, \nu(dq \times dp). 
\end{align*}
Here $\one_A$ is the indicator
function on the set $A$.

We now turn to characterizing convergence to
equilibrium for the system~\eqref{eqn:sdem}.  To define the weighted total variation norm which measures the convergence, for any $W\colon\mathcal{X}\rightarrow (0, \infty)$ measurable, we let $\mathcal{M}_W$
denote the set of probability measures $\nu$ on Borel subsets of $\mathcal{X}$ such that
\begin{align*}
\int_{\mathcal{X}} W(x) \, \nu(dx)< \infty.
\end{align*}
We equip the set $\mathcal{M}_W$ with the metric $\rho_W$ given by 
\begin{align}\label{eq:WTV}
\rho_W (\nu_1, \nu_2) = \sup_{\| \phi\|_W \leq 1}\int_{\mathcal{X}} \phi(x) (\nu_1(dx) - \nu_2(dx))
\end{align}
for any $\nu_1, \nu_2 \in \mathcal{M}_W$ where the seminorm $\| \cdot \|_W$ is defined for measurable $\phi\colon \mathcal{X}\rightarrow \R$ by
\begin{align*}
\| \phi \|_W = \sup_{x \in \mathcal{X}} \frac{|\phi(x)|}{1+W(x)}\,.
\end{align*}

With this setup, we will prove the following convergence result.
\begin{theorem}
\label{thm:main}
The standard Gibbs measure $\mu$ defined above in~\eqref{def:InvMeasure} is the unique invariant measure for the Markov process $x(t)=(q(t), p(t))$ on $\mathcal{X}$.  Furthermore for every $\b \in (0, \frac1{T})$, there exists $W\in C^\infty(\mathcal{X}; (0, \infty))$ and constants $C, \eta >0$ such that as $H(x)\rightarrow \infty$ with $x\in \mathcal{X}$  
\begin{align}
\label{eqn:Wasym}
W(x) = \exp\Big( b H(x)( 1+ o(1))\Big)
\end{align}
and such that for all $\nu \in \mathcal{M}_W$ and $t\geq 0$ 
\begin{align}
\label{eqn:geobound}
\rho_W(\nu P_t, \mu) \leq C e^{-\eta t} \rho_W(\nu, \mu).
\end{align}
\end{theorem}

\begin{remark}
Let $b\in (0, \frac1{T})$ and $W$ be the corresponding function given in the statement of the result above.  Observe that the asymptotic property~\eqref{eqn:Wasym} gives that $\mu \in \mathcal{M}_W$.  Furthermore, the result implies geometric convergence to equilibrium for any initial distribution $\nu$ on $\mathcal{X}$ with 
\begin{align*}
\int_{\mathcal{X}} e^{b H(x)}\,  \nu (dx)  < \infty.  
\end{align*}
Now restating the bound~\eqref{eqn:geobound} in terms of test functions $\phi\colon  (\R^d)^N\rightarrow \R$ with $\|\phi\|_W < \infty$, we find that there exists positive constants $C$ and $\eta$ (both depending on $b$) such that for any $x \in \mathcal{X}$ and any such $\phi$
\begin{align}
  \label{eq:testFunction}
 | P_t \phi(x) - \mu \phi | \leq C  \|\phi\|   e^{-\eta t}(1+ W(x)).
\end{align}
From~\eqref{eq:testFunction} and~\eqref{eqn:Wasym} one cannot expect this bound (and
hence Theorem~\ref{thm:main}) to hold if $b \geq \frac1T$ since for
such $b$, $\mu \phi = \infty$ for many functions $\phi$ with 
\begin{align*}
\sup_{x\in \mathcal{X}} \frac{|\phi(x)|}{\exp(b H(x))} < \infty. 
\end{align*}  
\end{remark}

In the process of proving Theorem~\ref {thm:main}, we also establish
the following result which is of independent interest.
\begin{proposition}
\label{prop:orbit}
For any $x \in \mathcal{X}$ and $t>0$,
$\supp P_t(x, \, \cdot \,) = \mathcal{X}$,
where $\supp \nu$ denotes the support of the measure $\nu$. 
For each $t>0$ and $x \in \mathcal{X}$, the measure $P_t(x, \,\cdot
\,)$ is absolutely continuous with respect to Lebesgue measure on
$\mathcal{X}$. Denoting the probability density of $P_t(x, \,\cdot
\,)$ by $r_t(x, y)$, the mapping $(t, x, y) \mapsto r_t(x, y)\colon (0, \infty) \times \mathcal{X} \times \mathcal{X}\rightarrow [0, \infty)$ is continuous. 
\end{proposition}

\section{Building intuition}
 \label{sec:heuristics}

Since  equation \eqref{eqn:sdem} generates a hypoelliptic diffusion
on $\mathcal{X}$, the main missing ingredient needed to apply the
results from \cite{HM_08} to prove exponential convergence is the existence of a
Lyapunov function (see also \cite{MSH_02, MTIII}).  While there are many form this could take, the most basic is a function $W\in C^2(\mathcal{X}; \R)$ satisfying $W\rightarrow \infty$ as $H\rightarrow \infty$ and the following global bound 
\begin{align}\label{eq:lyap}
  \mathcal{L} W \leq -c W + K
\end{align}
for some positive constants $c$ and $K$. In the above, $\mathcal{L}$ is the
generator of the SDE~\eqref{eqn:sdem} given by
\begin{align}
\label{eqn:generatorsimple}
\mathcal{L}= p \cdot \nabla_q -\nabla U(q)\cdot \nabla_p  - \gamma p \cdot \nabla_p +
  \gamma T \Delta_p=\mathcal{H} + \gamma \mathcal{R}\,.
\end{align} 
where 
\begin{align}
\label{eqn:HamOp}
\mathcal{H}= p \cdot \nabla_q - \nabla U(q) \cdot \nabla_p 
\end{align}
is the Liouville operator generated by the
Hamiltonian dynamics and 
\begin{align*}
\mathcal{R}= -p \cdot \nabla_p + T \Delta_p
\end{align*}
is the generator of an Ornstein-Uhlenbeck process.

A first natural guess for such a Lyapunov function $W$ would be $W=H$, as the Hamiltonian is conserved under the formal dynamics obtained by setting $\gamma=0$ in~\eqref{eqn:generatorsimple} and $\mathcal{R}$ is a very dissipative operator, though
only in the $p$ directions.  That is, applying the generator to $H$
we obtain
\begin{align}
\label{eqn:LHeq}
  \mathcal{L} H = - \gamma p^2 + \gamma T.
\end{align}
While the righthand side of the equation above is dissipative for large $p$, it is not
comparable to $H$ for $H\gg1$ if $p$ is order one and $U(q)$ is large.  How to transfer this dissipation from from $p$ to $q$ is the central issue we face.   

In \cite{cooke17:_geomet_ergod_two,HairerMattingly:2009} the authors exploit the fact that at high energies ($H\gg 1$), the system is well
approximated by the deterministic Hamiltonian dynamics obtained by
formally setting $\gamma=0$ in equation~\eqref{eqn:generatorsimple}.  Using this observation, one then averages the effect of the dissipation obtained in $p$ as in~\eqref{eqn:LHeq} around the deterministic Hamiltonian dynamics to obtain 
\begin{align*}
  \mathcal{L} H \approx -c H^q +\gamma T
\end{align*} 
for some constants $c>0$ and $q \geq 1$. To make this intuition rigorous,
a Lyapunov function $W$ of the form $W=H+\Psi$ is constructed where the lower-order perturbation $\Psi$ solves the
following Poisson equation
\begin{align}\label{H}
  \mathcal{H} \Psi(p,q) = \gamma p^2 - \gamma \langle p^2 \rangle(H(p,q))
\end{align}
where $\mathcal{H}$ is Liouville operator defined in equation~\eqref{eqn:HamOp} and $\langle p^2 \rangle(h)$ is the average of
$p^2$ along the deterministic Hamiltonian orbits, i.e. the characteristics of $\mathcal{H}$, at energy level $h$. 

In the case when $U(q)$ has no singular terms and is asymptotically equivalent to $|q|^\alpha$ as $|q|\rightarrow \infty$ for some constant $\alpha > 2$, one finds that the perturbation $\Psi$ satisfies
\begin{align*}
\Psi(p,q)\approx c_\alpha \, p\cdot q\,\, \text{ as } \,\, |q|\rightarrow \infty
\end{align*}
for some constant $c_\alpha>0$. Note that the relevance of $\mathcal{H}$ in equation~\eqref{H} above can be seen in this specific situation since, under the
scaling 
\begin{align*}
(p,q) \mapsto (\lambda p , \lambda^{\frac{2}{\alpha}}q),
\end{align*}
$\mathcal{L}$ behaves like $\lambda^{2-\frac{2}{\alpha}}\mathcal{H}$
for $\lambda\gg 1$. This scaling, in particular, provides the natural route
to infinity if one wishes to move along the constant $H$ energy
shells. 

If $U(q)$ however includes a singular term, for example taking the form
\begin{align}\label{eq:singularExample}
  U(q)=A |q|^\alpha + \frac{B}{|q|^\beta}
\end{align}
for some constants $\alpha>1$ and $A, B, \beta >0$, one must additionally consider the scaling 
\begin{align*}
(p,q) \mapsto
(\lambda p, \lambda^{-\frac{2}{\beta}} q)
\end{align*}
for $\lambda \gg 1$ in order to zoom in on the large $H$ region around $q =0$. One then
solves an equation like \eqref{H}, but with $\mathcal{H}$ replaced by
the correct dominant operator along different routes to
$H=\infty$ when $q\approx 0$. For further details in the case~\eqref{eq:singularExample}, see~\cite{cooke17:_geomet_ergod_two}.

While these ideas are appealing, the scalings used in
\cite{cooke17:_geomet_ergod_two, HairerMattingly:2009} always reduce
the problem to studying the deterministic Hamiltonian dynamics at
infinity. This can lead to a number of difficulties as the potential $U(q)$ becomes
complicated, leading to possibly chaotic dynamics.  For this reason, the
class of allowed potentials in \cite{cooke17:_geomet_ergod_two} is
very simple and in \cite{HairerMattingly:2009} the analysis is limited to short chains
of oscillators. On the other hand, the papers
\cite{AKM12, HerMat15} successfully make use of such scalings as they simplify the dynamics in the relevant limiting regions. Before returning to this
discussion, we make a short remark concerning hypercoercivity. 

\begin{remark}
Understanding why $V(p,q)=H(p,q) + c_\alpha p\cdot q$ above is a Lyapunov
function for the Langevin dynamics with a nonsingular potential in
\cite{MSH_02,Talay_2002} motivated the development of hypercoercivity as in~\cite{Villani_2009}.  A more recent set of works develops hypercoercivity in the $L^2$ setting~\cite{MR3324910,Grothaus_Stilgenbauer_2016} which simplifies the
analysis in many respects. This was then adapted to dynamics with
singular potentials in~\cite{Grothaus_Stilgenbauer_2015} which
extended and simplified the original approach in~\cite{Conrad_Grothaus_2010}.  However in both \cite{Conrad_Grothaus_2010, Grothaus_Stilgenbauer_2015}, the system must start in
equilibrium. Hence, the exponential decoration of time integrals when
starting in equilibrium and the general long-time existence of
solutions are the principle results of these
articles. In contrast, we will prove exponential  convergence to
equilibrium starting from any initial distribution in $\mathcal{M}_W$ where for any $b\in (0, \frac1{T})$, $W$ can be constructed to have the asymptotic property~\eqref{eqn:Wasym}. 
\end{remark}

\subsection{Motivating Scaling and the Overdamped Limit}
\label{sec:OverDamped}
We now explore a different scaling which generates different
paths to $H$ large, hence producing different asymptotic dynamics.
Inspired by the heuristics in Section 1.2 of
\cite{Dolbeault_Mouhot_Schmeiser_2010}, we consider the scaling 
\begin{align}\label{eq:scaling}
(t, q,p) \mapsto (\lambda^{-2} t, \lambda^{-1} q, p)
\end{align}
as $\lambda \rightarrow \infty$. Applying this scaling with $U(q) =
A |q|^\alpha$ to the Kolmogorov backward equation produces 
\begin{align*}
 \lambda\partial_t u = \nabla_p H \cdot \nabla_q u - \lambda^{-\alpha}
  \nabla_q H \cdot \nabla_p u -  \lambda^{-1}\gamma p\cdot \nabla_p u - 
 \lambda^{-1} \gamma  T\Delta_p u.  
\end{align*}
Thus if we rescale time by $t \mapsto t\lambda$ and consider a sequence of rescaled
potentials $U_\lambda(q) = A \lambda^\alpha |q|^\alpha$, then one can see that by formal asymptotics, the $\lambda \rightarrow \infty$ limit corresponds to the overdamped ($\gamma\rightarrow \infty$) limit. For a fixed potential $U$, the generator of the overdamped limit is given by 
\begin{align*}
  \mathcal{L}_0 = - \nabla_q U(q)\cdot \nabla_q + T\Delta_q .
\end{align*}
Ignoring the complication of the interaction of the time change with
the ever-steepening rescaled potential $U_\lambda$, the fact that we arrived at the
overdamped limit is encouraging. In particular, the overdamped dynamics has a
natural Lyapunov function, namely $U(q)$. Notice that this analysis does not
require any detailed understanding of the undamped, deterministic
Hamiltonian dynamics. This suggests that for large $|q|$, we expect
$q_t$ to follow the overdamped dynamics. A similar heuristic
analysis holds as $q \rightarrow 0$ with a different rescaling of the
potential with $U$ of the form~\eqref{eq:singularExample}. 

Motivated by this scaling analysis as well as the identity~\eqref{eqn:LHeq}, we next scale to large $U(q)$ while keeping $p$ unchanged. 

\subsection{Scaling at Large $U$}

To illustrate the basic idea, we consider the 
dynamics of a single particle in dimension
$d=1$ with $U(q)$ given by~\eqref{eq:singularExample}.  Because the particle cannot penetrate the origin, we set $U(q)=+\infty$ for $q\leq 0$ and leave $U(q)$ defined as in~\eqref{eq:singularExample} for $q>0$.  At the end of the calculation, we will remark how one can generalize what follows to the full setting of~\eqref{eqn:sdem} with $U$ simply being admissible.     

With this setup, one can easily check that $\mathcal{O}= \R_{>0}$ the potential $U$ is
admissible. Hence Proposition~\ref{prop:exist}  and Theorem~\ref{thm:main} apply. Both
proofs turn on the existence of an appropriate Lyapunov function, and the
most basic Lyapunov structure is that which is expressed in the inequality
\eqref{eq:lyap}. To obtain results in the weighted total variation
norm from equation \eqref{eq:WTV}, we will build a Lyapunov function
of the form 
\begin{align}
\label{eqn:Wform}
W=e^{\b V}\,\,\text{ where } \,\, V= H(1+ o(1))\,\text{ as } \, H\rightarrow \infty
\end{align}    
and $b>0$.  To see why we seek a $W$ of this form, applying the generator to $W$ produces
\begin{align*}
  \mathcal{L} W &=\b W \big[ \mathcal{L}V  + \b \gamma T|\nabla_p V|^2 \big]. 
\end{align*}
Thus if we can construct $V\in C^2(\mathcal{X})$ satisfying~\eqref{eqn:Wform} and 
\begin{align}
\label{eqn:inWineq}
\mathcal{L} V +   \b \gamma T |\nabla_p V |^2 \leq - C
\end{align}
on $\{W \geq R\}$ for some constants $C, R>0$, then
inequality~\eqref{eq:lyap} would hold for $W$ as $W$ is bounded on the sublevel set $\{ W< R\}$.  

To build intuition for how to find $V$ above, let us first simplify the problem above and attempt to find a weaker type of Lyapunov function $V_0\in C^2(\mathcal{X})$ such that $V_0 \rightarrow \infty$ as $H\rightarrow \infty$ and  
\begin{align}
\label{eqn:desiredboundmock}
\mathcal{L}V_0 \leq - C \,\,\text{ on } \,\,  \{ H \geq R\}  
\end{align}
where $C>0$ is a constant and $R>0$ is a sufficiently large constant.  We will see later that the stronger property~\eqref{eqn:inWineq} follows almost immediately by slightly tweaking $V_0$ satisfying the apparently weaker condition~\eqref{eqn:desiredboundmock}.

As mentioned previously, a first natural guess for any Lyapunov function $V_0$ is $V_0=H$ itself.  Clearly, $H\rightarrow \infty$ as $H \rightarrow \infty$, so turning to checking~\eqref{eqn:desiredboundmock} we find that by~\eqref{eqn:LHeq}, the desired
bound~\eqref{eqn:desiredboundmock} is satisfied when $p$ is large
enough but fails in regions of $\mathcal{X}$ where $p$ is sufficiently
small and $U(q)$ is large. This further motivates the choice of
scaling in equation \eqref{eq:scaling} as it keeps $p$ fixed and
scales $q$, hence only affecting the size of $U(q)$.

With this perspective, we concentrate on modifying $H$ to gain some
dissipation in the system when most of the energy is potential energy.
Intuitively, such effects are contained within the forcing term
$-\nabla U(q) \cdot \nabla_p$ in~\eqref{eqn:generatorsimple}.  To do so, we perturb off of our initial guess $V_0=H$ to produce a Lyapunov function $V_0$ of the form 
\begin{align*}
V_0= H+ \psi 
\end{align*}
where $\psi \in C^2(\mathcal{X})$ satisfies the following two qualitative properties:
\begin{itemize}
\item $\psi=o(H(q,p))$ as $H(q,p)\rightarrow \infty$, $(q,p) \in \mathcal{X}$;
\item $\psi$ provides a ``dissipative effect'' in $q$ resulting in the bound~\eqref{eqn:desiredboundmock}.  
\end{itemize}
To see how we should pick $\psi$, we start by analyzing the dynamics
when $p$ is bounded and $U(q)$ is large.  By combining the ideas of Section~\ref{sec:OverDamped} with those in the works~\cite{cooke17:_geomet_ergod_two, HerMat15}, we will scale the infinitesimal generator $\mathcal{L}$ while keeping only the dominant 
terms. Using this reduced infinitesimal generator will yield a simple PDE for the perturbation $\psi$ so that~\eqref{eqn:desiredboundmock} is satisfied.    

Given the structure of $U$ in equation~\eqref{eq:singularExample},
there are two qualitatively different routes to infinite energy in $q$:
$q\rightarrow \infty$ and $q \rightarrow 0^+$.  In light of this
observation, for $\lambda \gg1$ we will make the following two
substitutions on the generator $\mathcal{L} \colon (q,p) = (\lambda Q,P)$ to observe the large $q$ behavior and $(q,p)= (\lambda^{-1} Q_0, P_0)$ to see the small $q$ dynamics.  Note that, when written in the coordinates $(Q, P)$, since $d=N=1$ the generator has the form
\begin{align*}
\mathcal{L}_{(Q,P)} &= \lambda^{-1} P \partial_Q - \gamma P  \partial_P - U'(\lambda Q) \partial_P + \gamma T \partial_P^2\\
&\approx - \alpha A \lambda^{\alpha -1} Q^{\alpha-1}   \partial_P  .   
\end{align*}
On the other hand, when written in the coordinates $(Q_0, P_0)$ the generator satisfies
\begin{align*}
\mathcal{L}_{(Q_0, P_0)} &= \lambda P_0 \partial_{Q_0} - \gamma P_0 \partial_{P_0} - U'(\lambda^{-1} Q_0)\partial_{P_0} + \gamma T \partial_{P_0}^2\\
&\approx \beta B \lambda^{\beta+1} Q_0^{-\beta-1}  \partial_{P_0} .  
\end{align*}
Thus, in either of the two regimes when $U(q)$ is large and $p$ is bounded, the dominant balance of terms in $\mathcal{L}$ is encapsulated in the operator
\begin{align}
\mathcal{A}= - U'(q) \partial_p.  
\end{align}
Hence we seek a perturbation $\psi$ satisfying the PDE 
\begin{align}
\label{eqn:thePDE}
\mathcal{A} \psi(q,p) = - \kappa \,\, \text{ on } \,\,  \mathcal{X} \cap \{ U \geq R\} 
\end{align}
where $\kappa >0$ is a constant and $R>0$ is sufficiently large.  Since $|U'(q)| \rightarrow \infty$ as $U(q)\rightarrow \infty$, whenever $R>0$ is large enough
\begin{align*}
U'(q) \neq 0 \,\,\text{ on } \,\, \mathcal{X} \cap \{ U \geq R \}.  
\end{align*}
Thus the general solution of equation~\eqref{eqn:thePDE} is of the form
\begin{align}
\psi(q,p) = \kappa \frac{p}{U'(q)}  + \phi(q) \,\, \text{ on } \,\,  \mathcal{X} \cap \{U \geq R\}
\end{align}  
where $\phi$ is any function.  Picking $\phi\equiv 0$, it is clear that $\psi$ is asymptotically dominated by $H$ at large energies in the set $\mathcal{X} \cap \{U \geq R\}$ for $R>0$ large enough, as $U$ is admissible.  Assuming we have extended $\psi$ to be $C^2$ on $\mathcal{X}$, we will now see that $V_0=H+ \psi$ satisfies the bound~\eqref{eqn:desiredboundmock} on the restricted set $ \mathcal{X} \cap \{ U \geq R\}$ for $R>0$ large enough. 

First pick $\kappa= 2\gamma T$.  This choice implies that for $(q,p)\in   \mathcal{X} \cap \{ U \geq R\}$
\begin{align*}
\mathcal{L} V_0(q,p)&= \mathcal{L} (H + \psi)(q,p) = \mathcal{L} H(q,p) + \mathcal{A} \psi(q,p)+ (\mathcal{L}-\mathcal{A})\psi(q,p)\\
&= -\gamma p^2 -\gamma T    - 2\gamma T \frac{p^2 U''(q)}{(U'(q))^2} - 2\gamma^2 T \frac{p}{U'(q)}\\
& \leq  -\frac{3}{4}\gamma p^2 -\gamma T -2\gamma T\frac{p^2 U''(q)}{(U'(q))^2} + \frac{4\gamma^3 T^2}{(U'(q))^2}
\end{align*}
where the last line follows by Young's inequality applied to the term $-2\gamma^2 T p/ U'(q)$.  
Since $U$ is admissible, we note that we may pick $R>0$ large enough so that on $\mathcal{X} \cap \{U \geq R\}$ 
\begin{align*}
\frac{|U''(q)|}{(U'(q))^2} \leq \frac{1}{8 T}\,\, \text{ and } \,\, (U'(q))^2 \geq 8 \gamma^2 T.
\end{align*}
Combining this with the previous estimate implies that on $\mathcal{X} \cap \{U \geq R\}$
\begin{align}
\label{eqn:estV_0}
\mathcal{L} V_0(q,p)&\leq - \frac{1}{2} \gamma p^2 - \frac{1}{2}\gamma T\leq - \frac{1}{2}\gamma T.    
\end{align}
To produce the bound on the set $\mathcal{X} \cap \{ H \geq
R\}$, one can simply combine the estimate~\eqref{eqn:estV_0} with~\eqref{eqn:LHeq} by
setting $\psi=0$ on the set $\{ U\leq R'\}$ where $ R>R'\gg 1$ are
chosen appropriately and by smoothly interpolating between $H$ and
$H+\psi$ on their respective sets of definition.  

\begin{remark}
In the above, we have not chosen the constant $\kappa$ carefully.  With a little more care in Proposition~\ref{prop:main} of Section~\ref{sec:proof}, we will see that the value of $\b$ in $W=e^{bV}$ must be restricted to the interval $(0, \frac1{T})$, and that the choice of $\kappa$ will depend on the choice of $b$.  This choice of $\kappa$ will then allow us to set $V=V_0$ and conclude that $W=e^{bV}$ satisfies~\eqref{eq:lyap}.  Note that the restriction $b\in (0, \frac1{T})$ is natural from the point of view of the invariant distribution $\mu$, as any Lyapunov function $V$ satisfying~\eqref{eq:lyap} must belong to $L^1(\mu)$ (see, for example, \cite{HairerMattingly:2009, HerMat15}).  Thus in this sense the construction of $W$ is optimal.           
\end{remark}

We close this section by observing that there is a natural
generalization of this calculation to an arbitrary number of particles $(N\geq 1)$ in an arbitrary number of dimensions ($d\geq 1$).
\begin{remark}
In the full setting of~\eqref{eqn:sdem}, equation~\eqref{eqn:thePDE} generalizes to  
\begin{align}
\mathcal{A} \psi(q,p):=-\nabla U(q) \cdot \nabla_p \psi =  -\kappa\,\, \text{ on } \,\, \mathcal{X} \cap \{ U \geq R\}. 
\end{align}
Assuming that the potential $U\colon (\R^d)^N\rightarrow [0, \infty]$ is admissible, for $R>0$ large enough a particular solution of this equation is 
\begin{align}
\psi(q,p) =  \kappa \frac{p \cdot \nabla U(q)}{|\nabla U(q)|^2}.  
\end{align}
Extending $\psi$ to be $C^2$ on $\mathcal{X}$, we will see that $V_0=H+\psi$ satisfies 
\begin{align}
\mathcal{L} V_0 \leq -C \,\, \text{ on } \,\, \mathcal{X}\cap \{ H \geq R\} 
\end{align}
for $R>0$ large enough where $\mathcal{L}$ is the infinitesimal generator of the process~\eqref{eqn:sdem}.  
 \end{remark}

\section{Examples of Admissible Potentials}

Before proving Theorem~\ref{thm:main}, in this section we discuss in detail examples of admissible potentials $U$. We begin with a simple
$U$ which is absent of any singularities. Convergence results concerning~\eqref{eqn:sdem} under such conditions are certainly already covered by existing results (see, for example, \cite{Grothaus_Stilgenbauer_2016, MSH_02,Talay_2002}), but we include this example both as a warm-up and to show that Theorem~\ref{thm:main} subsumes these cases as well. Our second example returns to a single particle interacting with the origin via a singular potential, as was discussed in
Section~\ref{sec:heuristics} and studied in detail in~\cite{cooke17:_geomet_ergod_two}. In the final example, we consider the Lennard-Jones example mentioned in the introduction. 

\begin{example}
For simplicity, first suppose $U\in C^\infty( (\R^d)^N; [0, +\infty))$ is of the form 
\begin{align}
U(q)= \alpha |q|^{2k} + \phi(q)
\end{align}
where $\phi\colon (\R^d)^N \rightarrow \R$ is a polynomial of degree $j< 2k$.  In this case
\begin{align*}
 \mathcal{O}=\{ q\in (\R^d)^N \, : \, U(q)< \infty \} = (\R^d)^N.
 \end{align*}
Hence $\mathcal{O}$ is clearly open, path connected and the set $\{ q\in (\R^d)^N \, : \, U(q) < R\}$ is precompact for each $R>0$.  The finiteness condition~\eqref{eq:UIntegrable} is also clearly satisfied.  To see the needed asymptotic properties, first observe that there exist constants $c_1, C_i>0$ such that for all $|q| >0$ large enough 
\begin{align*}
| \nabla U(q)|\geq c_1|q|^{2k-1}- C_1 \qquad \,\, \text{ and } \,\, \qquad \frac{|\nabla^2 U(q)|}{|\nabla U(q)|^2}\leq \frac{C_3 |q|^{2k-2}+C_4}{|q|^{4k-2}}.  
\end{align*}
Hence, letting $\{q_k\}\subseteq (\R^d)^N$ be any sequence with $|q_k| \rightarrow \infty$ as $k\rightarrow \infty$, we have 
\begin{align*}
|\nabla U(q_k)| \rightarrow \infty \qquad \,\, \text{ and } \,\, \qquad \frac{|\nabla^2 U(q_k)|}{|\nabla U(q_k)|^2}\rightarrow 0
\end{align*}
as $k\rightarrow \infty$.  Thus, $U$ is an admissible potential.  
\end{example}

\begin{example}
\label{ex:oneparticle}
As we have discussed to some extent previously in Section~\ref{sec:heuristics}, assume in this example that $N=d=1$ and consider $U\colon\R\rightarrow [0, +\infty]$ of the form
\begin{align*}
U(q)= \begin{cases}
\displaystyle{Aq^\alpha + \frac{B}{q^\beta}} + \phi(q) & \text{ if } q>0\\
+\infty & \text{ if }\,\, q \leq 0
\end{cases}
\end{align*}
where $\alpha >1,\, A, B, \beta >0$ and $\phi\in C^\infty( \R_{>0} ; \R)$ is \emph{lower-order relative to} the leading order term $Aq^\alpha + Bq^{-\beta}$; that is, 
\begin{align*}
\lim_{q\rightarrow \infty} q^{-\alpha} \phi(q)= \lim_{q\rightarrow \infty} q^{1-\alpha} \phi'(q)= \lim_{q\rightarrow \infty} q^{2-\alpha} \phi''(q) = 0 
\end{align*} 
and
\begin{align*}
\lim_{q\rightarrow 0^+} q^\beta \phi(q)= \lim_{q\rightarrow 0^+} q^{\beta +1} \phi'(q)= \lim_{q\rightarrow 0^+} q^{\beta+2} \phi''(q) = 0. 
\end{align*} 
Equation~\eqref{eqn:sdem} with this choice of potential $U$ was studied extensively in~\cite{cooke17:_geomet_ergod_two}, but under the more restrictive conditions that $\alpha >2$ and that $\phi$ was a linear combination of powers of $q$ (including positive and negative) satisfying the above asymptotics.  We will see that in the more general setting considered above, $U$ is admissible.  Thus Theorem~\ref{thm:main} still applies.  

Here
\begin{align*}
\mathcal{O}= \{ q\in \R \, : \, U(q)< \infty\} = \R_{>0}
\end{align*}
is clearly open and path connected.  It is also easy to see that for each $R>0$ the set 
$\{ q\in \R \, : \, U(q) < R\}$ is precompact and the finiteness condition~\eqref{eq:UIntegrable} is satisfied.  Furthermore, by the asymptotics assumed of the lower-order perturbation $\phi$, for any sequence $q_k \in \R_{>0}$ with $q_k \rightarrow 0^+$ or $q_k \rightarrow \infty$ we have that 
\begin{align*}
|U'(q_k)|& = |A\alpha q_k^{\alpha-1} - B \beta q_k^{-\beta -1} + \phi'(q_k)|
&= \begin{cases}
A\alpha q^{\alpha -1} (1 + o(1)) &\text{ if } \,q_k\rightarrow \infty\\
B \beta q^{-\beta-1} (1 + o(1)) & \text{ if } \, q_k\rightarrow 0^+
\end{cases}. 
\end{align*}  
Thus in either case, $|U'(q_k)| \rightarrow \infty$ as $k\rightarrow \infty$.  Also observe that 
\begin{align*}
\frac{|U''(q_k)|}{|U'(q_k)|^2}&= \frac{|A\alpha(\alpha -1) q_k^{\alpha-2} + B \beta(\beta+1) q_k^{-\beta-2} + \phi''(q_k)|}{|A\alpha q_k^{\alpha-1} - B \beta q_k^{-\beta -1} + \phi'(q_k)|^2}\\
&= \begin{cases}
\frac{(\alpha-1)}{A\alpha} q_k^{-\alpha}(1+o(1))& \text{ if } q_k \rightarrow \infty\\
\frac{(\beta +1)}{B \beta } q_k^{\beta} (1+o(1)) & \text{ if } q_k \rightarrow 0^+\end{cases}.
\end{align*}
Thus in either case, $|U''(q_k)|/|U'(q_k)|^2\rightarrow 0$ as $k\rightarrow \infty.$  Hence, $U$ is an admissible potential.  
\end{example}

We next return to the Lennard-Jones example discussed in the introduction. Recall that this particular potential was the original motivation for this work.  
\begin{example}
\label{ex:singularH}
Consider now $N\geq 2$ particles in any dimension $d\geq 1$.  Here $U\colon(\R^d)^N\rightarrow [0, +\infty]$ will be of the form 
\begin{align*}
U(q)=\sum_{i=1}^N U_0(q_i) + \sum_{i< j} U_I(q_i -q_j)
\end{align*}  
where $q=(q_1, q_2, \ldots, q_N) \in (\R^d)^N$.  
We assume that $U_0 \in C^\infty(\R^d; [0, \infty))$ satisfies
\begin{align*}
U_0(x) = A |x|^{\alpha} + \phi_0(x) \,\, \text{ for } \,\, |x| \geq 1
\end{align*} 
where $A>0, \alpha >1$ and the perturbative part $\phi_0 \in C^\infty(\R^d)$ has the following asymptotic properties
\begin{align*}
\lim_{|x|\rightarrow \infty} |x|^{-\alpha} |\phi_0(x)|=\lim_{|x|\rightarrow \infty}  |x|^{1-\alpha} |\nabla \phi_0(x)| =\lim_{|x|\rightarrow \infty} |x|^{2-\alpha}| \nabla^2 \phi_0(x)|=0.
\end{align*}
For the interaction part of the potential, we have to be slightly careful in dimension $d=1$ to assure that the set $\mathcal{O}$ is path connected, as individual particles cannot pass one another.  In particular, in dimension $d=1$ we assume that $U_I \colon\R \rightarrow [0, +\infty]$ satisfies 
\begin{align*}
U_I(x)= \begin{cases}
B |x|^{-\beta} + \phi_I(x) & \text{ if } x < 0\\
+ \infty & \text{ if } x\geq 0
\end{cases} 
\end{align*}
and in dimension $d\geq 2$ we assume that $U_I\colon \R^d \rightarrow [0, +\infty]$ has 
\begin{align*}
U_I(x)= \begin{cases}
B |x|^{-\beta} + \phi_I(x) & \text{ if } x \neq 0\\
+ \infty & \text{ if } x= 0
\end{cases} .
\end{align*} 
In both expressions above, $B, \beta>0$ are constants.  The function $\phi_I$ is assumed to be $C^\infty$ on its respective domain ($x<0$ in $d=1$ and $x\neq 0$ in $d\geq 2$) and is assumed to satisfy the following asymptotic properties 
\begin{align*}
\lim_{x\rightarrow 0} |x|^\beta | \phi_I(x)|=\lim_{x\rightarrow 0}  |x|^{\beta+1} |\nabla \phi_I(x)| =\lim_{x\rightarrow 0} |x|^{\beta+2}| \nabla^2 \phi_I(x)|=0.
\end{align*}
By construction, note that 
\begin{align*}
\mathcal{O}= \begin{cases}
\big\{ (q_1, \ldots, q_N) \in \R^N \,: \, q_1< q_2 < \cdots < q_N \big\} & \text{ if } d=1\\
\big\{(q_1, \ldots, q_N) \in (\R^d)^N \,: \, q_i \neq q_j, \, i\neq j \big\} & \text{ if } d\geq 2
\end{cases}.  
\end{align*}
One can readily verify that $\mathcal{O}$ is open and path connected, while 
\begin{align*}
\{q\in (\R^d)^N \, : \,  U(q) < R\}\end{align*}
is precompact for every $R>0$.  The integrability condition~\eqref{eq:UIntegrable} is also satisfied.  While we expect the needed asymptotic properties of $|\nabla U|$ and $|\nabla^2 U|/|\nabla U|^2$ to be intuitively true given the outcome and structure of the potential in the previous example, they take a little more work to establish here due to the particle interactions.  The calculation which proves this, thus showing $U$ is admissible, is saved for the appendix (cf. Lemma~\ref{lem:singularH}).
\end{example}

\section{Proof of Theorem~\ref{thm:main}}
\label{sec:proof}

In this section we prove Theorem~\ref{thm:main}.  We begin by showing the following proposition, which is the main result on which the theorem relies.  The construction of the Lyapunov function below follows the heuristics outlined in Section~\ref{sec:heuristics}.  

\begin{proposition}
\label{prop:main}
For every $\b \in (0,\frac1{T})$, there exists $W \in C^{\infty}(\mathcal{X}; (0, \infty))$ satisfying the following two properties:
\begin{itemize}
\item[($\ell1$)] As $H(q,p) \rightarrow \infty$ with $(q,p) \in \mathcal{X}$
\begin{align*}
W(q,p) = \exp\big(\b H(q,p)( 1+ o(1))\big).
\end{align*}
In particular, $W(q,p) \rightarrow \infty$ as $H(q,p) \rightarrow \infty$, $(q,p) \in \mathcal{X}$.
\item[($\ell2$)]  There exist constants $c, K>0$ such that
\begin{align*}
\mathcal{L} W(q,p) \leq -c W(q,p) + K
\end{align*}
for all $(q,p)\in \mathcal{X}$.  Here we recall that $\mathcal{L}$ is the generator of the Markov process $(q(t), p(t))$ solving~\eqref{eqn:sdem}.     
\end{itemize}
\end{proposition}

\begin{proof}
For each $R_2>R_1>0$, let $\alpha \in C^\infty([0, \infty) ; [0,1])$ satisfy
\begin{align*}
\label{eqn:cutoff}
\alpha(x)= \begin{cases}
1 & \text{ if } x\geq R_2 \\
0 & \text{ if } x\leq R_1
\end{cases}
\,\,\, \text{ and } \,\,\,
|\alpha'| \leq \frac{2}{R_2 - R_1}.
\end{align*}
Fix $\b\in (0,\frac1{T})$ and $\kappa> 3 \gamma Nd$.  Since $U$ is admissible, we may pick $R_1>0$ large enough such that $|\nabla U(q)|\geq  1$ for $U\geq R_1/2$.  Consider a candidate Lyapunov functional $W: \mathcal{X}\rightarrow (0, \infty)$ defined by
\begin{align}
W(q,p) = \exp( \b H(q,p) + \psi(q,p))
\end{align}
where $\psi \in C^\infty(\mathcal{X} ; [0, \infty))$ is given by
\begin{align}
\psi(q,p) =\begin{cases}
 \displaystyle{\kappa   \, \alpha(U(q))    \frac{p\cdot \nabla U(q)}{|\nabla U(q)|^2} }& \text{ if } U(q) \geq R_1/2 \\
 0 & \text{ otherwise}
 \end{cases}.
\end{align}
First observe that $W\in C^{\infty}(\mathcal{X}; (0, \infty))$ and as $H(q,p) \rightarrow \infty$ with $(q,p) \in \mathcal{X}$
\begin{align*}
W(q,p) = \exp\big(\b H(q,p)(1+ o(1))\big).
\end{align*}
Thus property $(\ell 1)$ is satisfied.  To verify condition $(\ell 2)$, note that for $(q,p) \in \mathcal{X}$
\begin{align}
\label{eqn:m1}\frac{\mathcal{L} W(q,p)}{W(q,p)}& = - \b\gamma(1-\b T) |p|^2 - \kappa \alpha(U(q)) + p \cdot \nabla_q \psi(q,p) \\
\nonumber & \qquad + (2\b T-1)  \gamma \psi(q,p)+ \frac{\kappa^2 \gamma T\alpha^2(U(q)) }{|\nabla U(q)|^2}+ \gamma  \b T N d.
\end{align}
To estimate each of the terms on the righthand side of the equation above, first observe that 
\begin{align}
\nonumber p \cdot \nabla_q \psi(q,p) &=   \kappa \alpha(U(q)) \sum_{i=1}^d \sum_{\ell =1}^{N} p_\ell^i p \cdot  \partial_{q_\ell^i}\bigg( \frac{\nabla U(q)}{|\nabla U(q)|^2} \bigg) \\
\nonumber &\qquad +  \kappa \sum_{i=1}^d  \sum_{\ell =1}^{N} p_\ell^i p \cdot  \alpha'(U(q)) \partial_{q_\ell^i} U(q)  \frac{\nabla U(q)}{|\nabla U(q)|^2} \\
\label{eqn:T1}& \leq \kappa  |\alpha(U(q))|   |\nabla G(q)| |p|^2 + \kappa  |\alpha'(U(q))| |p|^2
\end{align}
where $G= \nabla U/ |\nabla U|^2$.  Also, for any $C>0$ Young's inequality gives
\begin{align}
\label{eqn:T2}
|\psi(q,p)|\leq \frac{\kappa}{2C} |p|^2 + \frac{\kappa C}{2}  \frac{\alpha^2(U(q))}{|\nabla U(q)|^2}.
\end{align}
Putting the estimates~\eqref{eqn:T1} and~\eqref{eqn:T2} into~\eqref{eqn:m1} yields the following bound
\begin{align*}
\frac{\mathcal{L} W(q,p)}{W(q,p)}& \leq - \bigg\{ \b \gamma(1 -\b T) - \kappa  \alpha(U(q)) |\nabla G(q)| - \kappa \alpha'(U(q)) - \frac{|2\b T-1|  \gamma \kappa}{2C}\bigg\} |p|^2 \\
& \qquad- \kappa \alpha(U(q))  + \Big(\frac{\kappa C}{2} |2\b T-1| \gamma  + \kappa^2 \gamma T\Big) \alpha^2(U(q)) |G(q)|^2  + \gamma \b T N d.
\end{align*}
Recalling that $\kappa> 0$ has already been chosen to satisfy $\kappa>3\gamma Nd$, we now start picking the remaining parameters $C, R_2>0$.  First, pick
\begin{align*}
C>\frac{4| 2\b T-1| \kappa}{\b (1-\b T)}.
\end{align*}
 Since $U$ is admissible, we may increase $R_1>0$ if necessary such that whenever $U(q) \geq R_1/2$ the following estimates must be satisfied
\begin{align*}
 |\nabla G(q)| &< \frac{\b \gamma (1- \b T)}{8 \kappa }, \,\,\,\text{ and} \,\,\, \Big( \frac{\kappa C}{2}|2\b T-1| \gamma  + \kappa^2\gamma T\Big) |G(q)|^2 \leq \gamma \b TNd .
 \end{align*}
Finally, pick $R_2>R_1$ such that
\begin{align*}
|\alpha'(U(q))|\leq \frac{\b \gamma(1-\b T)}{8\kappa}.
\end{align*}
Using the definition of $\alpha$ and putting these bounds together produces the global estimate
\begin{align*}
\frac{\mathcal{L} W(q,p)}{W(q,p)}& \leq -\frac{\b \gamma(1-\b T)}{2} |p|^2 -  \kappa \alpha(U(q))  + 2\gamma \b TNd
\end{align*}
on $\mathcal{X}$.  From this bound, note that if $U(q) \geq R_2$ we have
\begin{align*}
\frac{\mathcal{L} W(q,p)}{W(q,p)}&\leq -  \kappa   + 2\gamma \b TNd < - \gamma Nd
\end{align*}
as $\b < T^{-1}$.  On the other hand if $|p|^2  > 6\gamma Nd/(\b \gamma(1-\b T))$, we also have the previous estimate.  This now proves that there exist constants $c, K>0$ such that
\begin{align}
\mathcal{L}W\leq - c W + K
\end{align}
on $\mathcal{X}$, as $W$ is bounded on the set
\begin{align*}
\{(q,p) \in \mathcal{X} \, : \,  |p|^2 \leq 6 \gamma N d/(\b \gamma(1-\b T)) \} \cap \{ (q,p) \in \mathcal{X} \, : \, U(q) \leq R_2\}.
\end{align*}
This establishes property $(\ell 2)$ and hence the result.
\end{proof}

We next turn to the proof of Proposition~\ref{prop:orbit} which, via the support theorems~\cite{SV_72, SV_73}, relies on a deterministic control problem associated to the stochastic system~\eqref{eqn:sdem}.  To introduce it, consider the ODE on $\mathcal{X}$
\begin{align}
\label{eqn:cont}
\frac{dQ}{dt} &= P\\
\nonumber \frac{dP}{dt}&=- \gamma P - \nabla U(Q) + \sqrt{2\gamma T } \xi
\end{align}
where $\xi\in C([0, \infty); (\R^d)^N)$ is a control.  For $x_0 \in \mathcal{X}$ and $t>0$, we define $\mathcal{X}_t(x_0)$ to be the set of points $x \in \mathcal{X}$ such that there exists $\xi\in C([0, \infty); (\R^d)^N)$ for which the solution of~\eqref{eqn:cont} is defined on $[0, t]$ and satisfies $(Q(0), P(0))=x_0$, $(Q(t), P(t))= x$ and
\begin{align}
\label{eqn:finitepot}
U(Q(s)) < \infty\,\, \text{ for all }\,\, 0 \leq s\leq t.
\end{align}

\begin{proof}[Proof of Proposition~\ref{prop:orbit}]
By the support theorems~\cite{SV_72, SV_73}, in order to prove that $\supp P_t(x_0, \cdot) = \mathcal{X}$ for any $x_0 \in \mathcal{X}$ and any $t>0$ and that $\supp \nu = \mathcal{X}$ for any invariant measure $\nu$ for $P_t$, it suffices to show that 
\begin{align*}
\mathcal{X}_t(x_0) = \mathcal{X}\qquad \forall x_0 \in \mathcal{X}, \, \forall t>0.  
\end{align*} 
To do so, first fix $x_0=(q_0, p_0), x=(q,p) \in \mathcal{X}$ and $t>0$.  As $\mathcal{O}$ is path connected, so is $\mathcal{X}= \mathcal{O}\times (\R^d)^N$.  Thus there exists a continuous path contained in $\mathcal{X}$ connecting $x_0$ and $x$.  We now use this path to construct an associated path that solves the control problem~\eqref{eqn:cont} at time $t$ moving from the initial condition $x_0$ to the desired target $x$.  Note that since $\mathcal{X}$ is open in $(\R^d)^N\times (\R^d)^N$, we may approximate this continuous path by $\zeta \in C^\infty([0, \infty); \mathcal{X})$ with $\zeta(0)=x_0$ and $\zeta(t)=x$.  Let $\psi\in C^\infty([0, \infty); \mathcal{O})$ denote the $q$-coordinate of the path $\zeta$.  By picking $\epsilon >0$ below small enough, using $\psi$ one can construct $\phi \in C^\infty([0,\infty); \mathcal{O})$ satisfying 
\begin{align*}
\phi(s)= \begin{cases}
q_0+ s p_0 & \text{ if } 0 \leq s \leq \epsilon\\
q +(s-t) p & \text{ if } t-\epsilon \leq s \leq t
\end{cases}.
\end{align*}
Consider now $(Q, P)\in C^\infty([0, \infty); \mathcal{X})$ defined by $Q=\phi$ and $P= \phi'$.  Then this choice of $(Q,P)$ solves the control problem~\eqref{eqn:cont} with $(Q(0), P(0))=(q_0, p_0)$ and $(Q(t), P(t))=(q,p)$ by picking
\begin{align*}
\xi= \frac{1}{\sqrt{2\gamma T}} \big(\phi'' + \gamma \phi'+ \nabla U(\phi) \big).
\end{align*}
We note that, carefully, since $\phi$ maps into $\mathcal{O}$ then $U(\phi(s))< \infty$ for all $0\leq s \leq t$.  This finishes the proof that $\mathcal{X}_t(x_0)=\mathcal{X}$ for all $x_0 \in \mathcal{X}$ and all $t>0$.

The remaining claims in the lemma concerning the existence and regularity of a probability density function $r_t(x,y)$ follow by H\"{o}rmander's hypoellipticity theorem~\cite{Hor_67}.  Indeed, one can readily check that since $U\in C^\infty(\mathcal{O})$, standard commutator calculations show that the the operators $\partial_t \pm \mathcal{L}$, $\partial_t \pm \mathcal{L}^*$, $\mathcal{L}$, $\mathcal{L}^*$ are hypoelliptic on the respective open sets $(0, \infty) \times \mathcal{X}$, $(0, \infty) \times \mathcal{X}$, $\mathcal{X}$, $\mathcal{X}$ where $\mathcal{L}^*$ denotes the formal $L^2$ adjoint of $\mathcal{L}$.
\end{proof}

Before concluding Theorem~\ref{thm:main}, we need the following two
corollaries.  These will allow us to more easily connect with and
apply the main result of \cite{HM_08} (or results in \cite{MTIII} ). The first is a translation of
the Lyapunov bound in $(\ell 2)$ of Proposition~\ref{prop:main} to a bound on
the Markov semigroup, while the second result establishes the needed
minorization condition. As both are relatively standard, the proofs can be found in the Appendix.

\begin{corollary}
\label{cor:1}
Fix $\b \in (0, T^{-1})$ and let $W$ be the Lyapunov function given in Proposition~\ref{prop:main} corresponding to this choice of $b$.  Also, let $c, K>0$ be the corresponding constants given in Proposition~\ref{prop:main} $(\ell2)$. Then for all $t\geq 0$ and $x\in \mathcal{X}$
\begin{align}
P_t W(x) \leq e^{-c t } W(x) + K/c .
\end{align}
\end{corollary}

\begin{corollary}
\label{cor:2}
Fix $\b \in (0, T^{-1})$ and let $W$ denote the Lyapunov function assured by Proposition~\ref{prop:main} for this choice of $b$.  For $R>0$, define the set $\mathcal{C}_R=\{ x \in \mathcal{X} \, : \, W(x) \leq R \}.$  Then for each $R>0$, $\mathcal{C}_R$ is compact.  Furthermore, for each $R>0$ large enough and each $t_0>0$ there exists a Borel probability measure $\nu$ on $\mathcal{X}$ and $\alpha>0$ such that for all $A\subseteq \mathcal{X}$ Borel and all $x\in \mathcal{C}_R$
\begin{align}
P_{t_0}(x,A) \geq \alpha \, \nu(A).
\end{align}
\end{corollary}

We are now ready to prove Theorem~\ref{thm:main}.

\begin{proof}[Proof of Theorem~\ref{thm:main}]
Fix $\b \in (0, T^{-1})$ and $t_0>0$.  To see that the Gibbs measure
$\mu$ is invariant for the Markov process $x(t)$ on $\mathcal{X}$, one
needs only check that
\begin{align*}
\mathcal{L}^*( \exp(-T^{-1} H))(x) =0
\end{align*}
for all $x\in \mathcal{X}$, where we recall that $\mathcal{L}^*$ is the formal $L^2$-adjoint of the generator $\mathcal{L}$.  Uniqueness of the invariant measure then follows easily by ergodic decomposition, as $P_t$ is strong Feller and $\supp \nu = \mathcal{X}$ for any invariant measure $\nu$ for $x(t)$ by Proposition~\ref{prop:orbit} (see, for example, Theorem 4.2.1 in~\cite{DPZ_96} or Corollary~3.17 of~\cite{HM_06}).

We next seek to apply Theorem~1.2 of \cite{HM_08} to the embedded Markov chain on $\mathcal{X}$ given by $\mathcal{P}_n=P_{n t_0}$.  Note that Corollary~\ref{cor:1} and Corollary~\ref{cor:2} imply Assumption~1 and Assumption~2 of~\cite{HM_08}.  Thus applying Theorem~1.2 of \cite{HM_08}, we find that there exist constants $C >0$ and $\delta \in (0,1)$ such that
\begin{align}
\label{eqn:discbound}
\rho_W(\mathcal{P}_n \nu_1, \mathcal{P}_n \nu_2) \leq C \delta^{n} \rho_W(\nu_1, \nu_2)
\end{align}
for all $n\in \N\cup \{ 0\}$ and all $\nu_i \in \mathcal{M}_W$.  We next convert the bound~\eqref{eqn:discbound} to general times $t\geq 0$.  Let $t= nt_0 + \epsilon$ where $n\in \N\cup \{ 0\}$ and $\epsilon \in (0, t_0)$.  Observe that for any $\phi\colon \mathcal{X}\rightarrow \R$ measurable with $\|\phi\|_W\leq 1$, we have by Corollary~\ref{cor:1} 
\begin{align*}
\|P_\epsilon \phi\|_W &\leq \| \phi\|_W \sup_{x \in \mathcal{X}} \frac{1 +\E_{x} W(x(\epsilon))}{1+W(x)}\leq C'
\end{align*}
for some constant $C'>0$ which is independent of $\epsilon >0$.  By Fubini and the Chapman-Kolmogorov equations, one can then check that
\begin{align*}
\rho_W(P_t \nu_1, P_t \nu_2) = \rho_W( P_\epsilon \mathcal{P}_n \nu_1, P_\epsilon \mathcal{P}_n \nu_2)\leq C' \rho_W(\mathcal{P}_n \nu_1, \mathcal{P}_n \nu_2)\leq C C' \delta^n \rho_W(\nu_1, \nu_2).
\end{align*}
Picking $\eta= - t_0^{-1} \log \delta$ and $C''= C C'/ \delta^{1/t_0}$ we obtain the bound
\begin{align*}
\rho_W(P_t \nu_1, P_t \nu_2) \leq C'' e^{-\eta t } \rho_W(\nu_1, \nu_2)
\end{align*}
which is satisfied for all $t\geq 0$ and all $\nu_i \in \mathcal{M}_W$.  Since $\mu \in \mathcal{M}_W$ and $P_t \mu= \mu$, the result now follows.
\end{proof}

\section*{Appendix}

We have left to prove that $U$ defined in Example~\ref{ex:singularH} is an admissible potential and to establish Corollary~\ref{cor:1} and Corollary~\ref{cor:1}. We begin with:

\begin{lemma}
\label{lem:singularH}
Consider the potential $U$ and the open set $\mathcal{O}$ defined in Example~\ref{ex:singularH}.  Then there exist constants $c_i, D>0$ such that
\begin{align}
\label{eqn:lowergrad}
| \nabla U(q)| \geq c_1 |q|^{\alpha -1} + c_2\sum_{i<j} |q_i-q_j|^{-\beta-1} -D
\end{align}
for all $q\in \mathcal{O}$.  Consequently, $U$ is an admissible potential.
\end{lemma}

The argument is a slight reworking of the proof of Lemma~4.12 of~\cite{CG_10}.  
\begin{proof}
From the bound~\eqref{eqn:lowergrad}, the asymptotic properties in the third part of the definition of an admissible potential are immediate, which was all that left to show that $U$ is admissible.  To establish the claimed bound, by the asymptotic properties of $\phi_0$ and $\phi_I$, note that it suffices to assume that $\phi_0=\phi_I=0$.  The basic idea behind the rest of the proof is to pick the right directional derivative.  Notationally, set $\mathcal{Z}_N= \{ 1, 2, \ldots, N\}$.

We first claim that there exists constants $d_i, D_i >0$ such that
\begin{align}
\label{eqn:bound1}
|\nabla U(q)| \geq  d_i |q_i|^{\alpha-1} - D_i
\end{align}
for all $i=1,2, \ldots, N$ and $q\in \mathcal{O}$.  From this, it follows immediately that
\begin{align}
\label{eqn:bound2}
|\nabla U (q)| \geq 2c_1 |q|^{\alpha-1} -C_1
\end{align}
on $\mathcal{O}$ for some constants $c_1, C_1 >0$.  Note that without loss of generality it suffices to show the bound~\eqref{eqn:bound1} above for $i=1$ and for $|q_1| > 1$.  For $q=(q_i)_{i=1}^N\in \mathcal{O}$, consider an increasing sequence of sets $S_i(q)$, $i=1,2, \ldots, N$, defined inductively as follows:
\begin{align*}
S_1(q) &= \{ j \in \mathcal{Z}_N\, : \, |q_1- q_j| < N^{-1}\} \\
S_m(q) &= \{ j \in \mathcal{Z}_N \, : \, |q_j - q_k|< N^{-1} \,\, \exists \,\,k \in S_{m-1}(q)\}, \,\,\,\, m =2, \ldots, N.
\end{align*}
Let $\mathcal{S}(q)= S_N(q)$ and observe that for any $i,j \in \mathcal{S}(q)$, $|q_i - q_j| < 1$.  Consequently, combining $|q_1-q_j|^2=|q_1|^2+ |q_j|^2 - 2 q_1\cdot q_j$ with the inequality $|q_1-q_j|<1$, it follows that for $j\in \mathcal{S}$
\begin{align*}
q_1 \cdot q_j \geq \frac{|q_1|^2 + |q_j|^2 - 1}{2} \geq 0
\end{align*}
where the last inequality follows since $|q_1|>1$.  
Moreover, if $i\in \mathcal{S}(q)$ while $j\notin \mathcal{S}(q)$ we have that $|q_i - q_j | \geq N^{-1}$.  Let $\sigma(q)=(\sigma_1(q), \ldots, \sigma_N (q)) \in (\R^d)^N$ be such that $\sigma_i(q)= q_1/|q_1|$ if $i\in \mathcal{S}(q)$ and $\sigma_i=0$ otherwise.  We thus have the bound
\begin{align*}
&\sqrt{N} |\nabla U(q)| \\
&\geq \sigma(q) \cdot \nabla U(q)\\
&= \sum_{k \in \mathcal{S}} A \alpha \frac{q_1\cdot q_k}{|q_1|}|q_k|^{\alpha-2}  + \sum_{k\in \mathcal{S}} \sum_{\substack{j=1\\j\neq k}}^N B\beta \frac{q_1 |q_1|^{-1}\cdot (q_j - q_k)}{|q_j-q_k|^{\beta+2}}\\
&\geq A\alpha |q_1|^{\alpha-1} + \sum_{k\in \mathcal{S}} \sum_{\substack{j=1\\j\neq k\\j \in \mathcal{S}}}^N B\beta \frac{q_1 |q_1|^{-1}\cdot (q_j - q_k)}{|q_j-q_k|^{\beta+2}} + \sum_{k\in \mathcal{S}} \sum_{\substack{j=1\\j\neq k\\j \notin \mathcal{S}}}^N B\beta \frac{q_1 |q_1|^{-1}\cdot (q_j - q_k)}{|q_j-q_k|^{\beta+2}} \\
&\geq  A\alpha |q_1|^{\alpha-1} +0 - B\beta N^{\beta + 3}.
\end{align*}
This finishes the proof of the bound~\eqref{eqn:bound1} when $i=1$, as desired.

We next show that for all $i,j \in \mathcal{Z}_N$ with $i\neq j$, there exist constants $d_{i,j}, D_{i,j}>0$ such that
\begin{align}
\label{eqn:bound3}
|\nabla U(q)| \geq d_{i,j} |q_i-q_j|^{-\beta-1} - D_{i,j}
\end{align}
for all $q =(q_1, q_2, \ldots, q_N) \in \mathcal{O}$.  From this estimate and the bound~\eqref{eqn:bound2}, the lemma thus follows.  Without loss of generality, we will prove the bound~\eqref{eqn:bound3} for $i=1, j=2$.  For $q\in \mathcal{O}$, let $\sigma(q)= (q_2-q_1)/|q_2-q_1|$ and $\xi_k(q)= c_k(q) \sigma(q)$ where the constants $c_k(q) \in \{-1,1\}$ are chosen as follows: $c_k(q) = 1$ if $q_k\cdot \sigma < q_2 \cdot \sigma$ and $c_k(q)=-1$ otherwise.  With this choice of direction $\xi(q):=(\xi_1(q), \ldots, \xi_N(q))$ we find that on $\mathcal{O}$:
\begin{align}
\label{eqn:bound4}
\nonumber \sqrt{N} |\nabla U(q)|& \geq \xi(q) \cdot \nabla U(q)\\
\nonumber & = \sum_{k =1}^N A\alpha \xi_k(q) \cdot q_k |q_k|^{\alpha-2}  + \sum_{j <k}  B \beta (\xi_j(q)-\xi_k(q)) \cdot \frac{q_k-q_j}{|q_k-q_j|^{\beta+2}}\\
&\geq -  D_2' |q|^{\alpha-1} + B\beta |q_1 -q_2|^{-\beta-1}
\end{align}
for some constant $D_2'>0$.  Combining the bound~\eqref{eqn:bound2} with~\eqref{eqn:bound4} we arrive at the desired estimate.
\end{proof}

We next finish proving Corollary~\ref{cor:1} and Corollary~\ref{cor:2}.

\begin{proof}[Proof of Corollary~\ref{cor:1}]
Fix $b\in (0, \frac1{T})$ and let $W$ and $c,K>0$ be, respectively, the corresponding Lyapunov function and constants given in Proposition~\ref{prop:main}.  Let $\tau_n= \inf\{ t>0 \, : \, W(x(t)) \geq n \}$ and $\tau_n(t)= \tau_n \wedge t$.  Define 
\begin{align*}
V(t, x) = e^{c t} (W(x)-K/c).  
\end{align*}
It follows by Ito's formula that
\begin{align*}
\E_x V (\tau_n(t), x(\tau_n(t))) \leq W(x),
\end{align*}
implying
\begin{align*}
\E_x e^{c \tau_n(t)} W (x(\tau_n(t))) \leq W(x) + e^{c t} K/c.
\end{align*}
Since $W\geq 0$, we may estimate the lefthand side of the above as
\begin{align*}
\E_x e^{c \tau_n(t)} W (x(\tau_n(t)))\geq e^{c t} \E_x \textbf{1}_{\{ t < \tau_n\}} W(x(t)).
\end{align*}
Combining the previous two estimates and applying monotone convergence as $n\rightarrow \infty$ gives the desired bound since $\PP_x\{ \lim_{n\rightarrow \infty} \tau_n > t\}=1$ by Proposition~\ref{prop:exist}.
\end{proof}

\begin{remark}
  The proof of Proposition~\ref{prop:exist} does not require us to first
  find a Lyapunov function $W$ of the type in Proposition~\ref{prop:main} or the strong exponential  moment bound in Corollary~\ref{cor:1}. A simple stoping time argument using $H$ as a test function instead of $W$ can produce $\E H(x_t) \leq H(x_0) + C_0 t$ and 
  \begin{align*}
    \E \sup_{s\leq t} H(x_s)\leq H(x_0) + C_1 t
  \end{align*}
for postitive constants $C_0$ and $C_1$. This is to sufficient derive
global existence. 
\end{remark}

\begin{proof}[Proof of Corollary~\ref{cor:2}] Here we largely follow
  the path laid out in \cite{MSH_02}.
First note that for each $R>0$, $\mathcal{C}_R$ is compact by definition of an admissible potential and Proposition~\ref{prop:main}.  Also, since $W(x)\rightarrow \infty$ as $H(x)\rightarrow \infty$, $x\in \mathcal{X}$, we may pick $R_0>0$ such that for $R\geq R_0$, $\text{interior}(\mathcal{C}_R) \neq \emptyset$.  Thus fix $R\geq R_0$ and $t_0>0$.  By Proposition~\ref{prop:orbit}, we have that
\begin{align*}
P_t(x, B_\delta(y)) >0
\end{align*}
for all $x,y \in \mathcal{X}$ and all $\delta >0$, where $B_\delta(y)$ denotes the open ball of radius $\delta>0$ centered at $y$.  Thus by Proposition~\ref{prop:orbit} again, pick $(y^*, z^*) \in \text{interior}(\mathcal{C}_R)$ such that
\begin{align*}
r_{t_0/2}(y^*, z^*)=2\epsilon >0.
\end{align*}
Also, by continuity of the density $r_t(x,y)$ we may pick $\delta >0$ such that $B_\delta(y^*) \times B_\delta(z^*) \subseteq \text{interior}(\mathcal{C}_R) \times \text{interior}(\mathcal{C}_R)$ and such that for all $(y, z) \in B_\delta(y^*)\times B_\delta(z^*)$\begin{align*}
r_{t_0/2}(y, z)\geq \epsilon.
\end{align*}
For $A\subseteq \mathcal{X}$ Borel, define
\begin{align*}
\nu(A)= \frac{1}{|B_\delta(z^*)|} |A \cap B_\delta(z^*)|
\end{align*}
where in the above $| \,\cdot\, |$ denotes Lebesgue measure on Borel subsets of $(\R^d)^N \times (\R^d)^N$.  Then observe that for all $x\in \mathcal{C}_R$ and $A\subseteq \mathcal{X}$ Borel
\begin{align*}
P_{t_0}(x, A) &\geq \int_{A\cap B_\delta(z^*)} \int_{B_\delta(y^*)} r_{t_0/2}(x,y) r_{t_0/2}(y,z) \, dy \, dz \\
&\geq \epsilon |B_\delta(z^*)|  \inf_{x\in \mathcal{C}_R} P_{t_0/2}(x, B_\delta(y^*))\nu(A)
\end{align*}
Since $P_{t_0/2}(x, B_\delta(y^*)) >0$ for all $x\in \mathcal{C}_R$ and $x\mapsto P_{t_0/2}(x, B_\delta(y^*))$ is continuous by Proposition~\ref{prop:orbit}, the result now follows since $ \inf_{x\in \mathcal{C}_R} P_{t_0/2}(x, B_\delta(y^*))>0$.  \end{proof}

\section*{Acknowledgements}
We are grateful to Scott A. McKinley and Gabriel Stoltz for fruitful
conversations about the topic of this paper. In particular, Stoltz was
part of early discussions around the possible relationship between
\cite{Grothaus_Stilgenbauer_2016,MR3324910} and building a Lyapunov
function for Hamiltonian dynamics with singular potentials. D.P.H. and
J.C.M. are respectively supported in part by grant DMS-1612898 and
DMS-1613337 from the National Science Foundation.

\bibliographystyle{plain}
\bibliography{gham}

\begin{thebibliography}{10}

\bibitem{AKM12}
Avanti Athreya, Tiffany Kolba, and Jonathan~C. Mattingly.
\newblock Propagating {L}yapunov functions to prove noise-induced
  stabilization.
\newblock {\em Electron. J. Probab.}, 17:no. 96, 38, 2012.

\bibitem{Conrad_Grothaus_2010}
Florian Conrad and Martin Grothaus.
\newblock Construction, ergodicity and rate of convergence of n-particle
  langevin dynamics with singular potentials.
\newblock {\em Journal of Evolution Equations}, 10(3):623–662, Aug 2010.

\bibitem{CG_10}
Florian Conrad and Martin Grothaus.
\newblock Construction, ergodicity and rate of convergence of {$N$}-particle
  {L}angevin dynamics with singular potentials.
\newblock {\em J. Evol. Equ.}, 10(3):623--662, 2010.

\bibitem{cooke17:_geomet_ergod_two}
Ben Cooke, David~P. Herzog, Jonathan~C .~Mattingly, Scott~A. McKinley, and
  Scott~C. Schmidler.
\newblock Geometric ergodicity of two--dimensional hamiltonian systems with a
  lennard--jones--like repulsive potential.
\newblock {\em Communications in Mathematical Science}, 2017.

\bibitem{DPZ_96}
G.~Da~Prato and J.~Zabczyk.
\newblock {\em Ergodicity for infinite-dimensional systems}, volume 229 of {\em
  London Mathematical Society Lecture Note Series}.
\newblock Cambridge University Press, Cambridge, 1996.

\bibitem{MR3324910}
Jean Dolbeault, Cl\'ement Mouhot, and Christian Schmeiser.
\newblock Hypocoercivity for linear kinetic equations conserving mass.
\newblock {\em Trans. Amer. Math. Soc.}, 367(6):3807--3828, 2015.

\bibitem{Dolbeault_Mouhot_Schmeiser_2010}
Jean Dolbeault, Clément Mouhot, and Christian Schmeiser.
\newblock Hypocoercivity for linear kinetic equations conserving mass.
\newblock {\em arXiv:1005.1495 [math]}, May 2010.
\newblock arXiv: 1005.1495.

\bibitem{Grothaus_Stilgenbauer_2015}
Martin Grothaus and Patrik Stilgenbauer.
\newblock A hypocoercivity related ergodicity method with rate of convergence
  for singularly distorted degenerate kolmogorov equations and applications.
\newblock {\em Integral Equations and Operator Theory}, 83(3):331–379, Nov
  2015.
\newblock arXiv: 1506.04386.

\bibitem{Grothaus_Stilgenbauer_2016}
Martin Grothaus and Patrik Stilgenbauer.
\newblock Hilbert space hypocoercivity for the langevin dynamics revisited.
\newblock {\em arXiv:1608.07889 [math-ph]}, Aug 2016.
\newblock arXiv: 1608.07889.

\bibitem{HM_06}
Martin Hairer and Jonathan~C. Mattingly.
\newblock Ergodicity of the 2{D} {N}avier-{S}tokes equations with degenerate
  stochastic forcing.
\newblock {\em Ann. of Math. (2)}, 164(3):993--1032, 2006.

\bibitem{HairerMattingly:2009}
Martin Hairer and Jonathan~C. Mattingly.
\newblock Slow energy dissipation in anharmonic oscillator chains.
\newblock {\em Comm. Pure Appl. Math.}, 62(8):999--1032, 2009.

\bibitem{HM_08}
Martin Hairer and Jonathan~C. Mattingly.
\newblock {\em Yet Another Look at Harris' Ergodic Theorem for Markov Chains},
  pages 109--117.
\newblock Springer Basel, Basel, 2011.

\bibitem{HerMat15}
David Herzog and Jonathan Mattingly.
\newblock Noise-induced stabilization of planar flows i.
\newblock {\em Electron. J. Probab.}, 20:43 pp., 2015.

\bibitem{Hor_67}
Lars H\"ormander.
\newblock Hypoelliptic second order differential equations.
\newblock {\em Acta Math.}, 119:147--171, 1967.

\bibitem{Kh_12}
Rafail Khasminskii.
\newblock {\em Stochastic stability of differential equations}, volume~66 of
  {\em Stochastic Modelling and Applied Probability}.
\newblock Springer, Heidelberg, second edition, 2012.
\newblock With contributions by G. N. Milstein and M. B. Nevelson.

\bibitem{MSH_02}
J.~C. Mattingly, A.~M. Stuart, and D.~J. Higham.
\newblock Ergodicity for {SDE}s and approximations: locally {L}ipschitz vector
  fields and degenerate noise.
\newblock {\em Stochastic Process. Appl.}, 101(2):185--232, 2002.

\bibitem{MTIII}
Sean~P. Meyn and R.~L. Tweedie.
\newblock Stability of markovian processes iii: Foster-lyapunov criteria for
  continuous-time processes.
\newblock {\em Advances in Applied Probability}, 25(3):518--548, 1993.

\bibitem{RB_06}
Luc Rey-Bellet.
\newblock Ergodic properties of {M}arkov processes.
\newblock In {\em Open quantum systems. {II}}, volume 1881 of {\em Lecture
  Notes in Math.}, pages 1--39. Springer, Berlin, 2006.

\bibitem{SV_72}
Daniel~W. Stroock and S.~R.~S. Varadhan.
\newblock On the support of diffusion processes with applications to the strong
  maximum principle.
\newblock pages 333--359, 1972.

\bibitem{SV_73}
Daniel~W. Stroock and S.~R.~S. Varadhan.
\newblock Probability theory and the strong maximum principle.
\newblock pages 215--220, 1973.

\bibitem{Talay_2002}
D.~Talay.
\newblock Stochastic hamiltonian systems: exponential convergence to the
  invariant measure, and discretization by the implicit euler scheme.
\newblock {\em Markov Processes and Related Fields}, 8(2):163–198, 2002.
\newblock Inhomogeneous random systems (Cergy-Pontoise, 2001).

\bibitem{Villani_2009}
C\'edric Villani.
\newblock Hypocoercivity.
\newblock {\em Mem. Amer. Math. Soc.}, 202(950):iv+141, 2009.

\end{thebibliography}
\end{document}